\documentclass[11pt,twoside]{article}

\usepackage{mathrsfs,amsfonts,amsmath,amssymb}
\usepackage{latexsym}
\usepackage{comment} 
\newtheorem{satz}{Theorem}

\newtheorem{theorem}[satz]{Theorem}
\newtheorem{lemma}[satz]{Lemma}

\def\Z{\mathbb {Z}}
\def\F{\mathbb {F}}

\def\a{\alpha}

\def\d{\delta}

\def\({\big (}
\def\){\big )}

\def\G{\Gamma}

\def\le{\leqslant}
\def\ge{\geqslant}
\def\_phi{\varphi}
\def\eps{\varepsilon}

\def\Gr{{\mathbf G}}

\def\t{\tilde}

\def\Cf{{\mathcal C}}
\def\la{\lambda}
\def\D{\Delta}

\def\supp{\mathsf{supp}}

\def\T{\mathsf{T}}

\def\SL{{\rm SL}}

\def\Stab{{\rm Stab}}

\newcommand{\zk}{\ensuremath{\mathfrak{k}}}

\newenvironment{dedication}
  {
   \thispagestyle{empty}
   \itshape             
   \raggedleft          
  }
  {\par 
  }

\author{N.G. Moshchevitin, B. Murphy and I.D. Shkredov}
\title{On Korobov bound concerning Zaremba's conjecture 
}
\date{}
\begin{document}
	\maketitle  
	\begin{dedication}
    \`A Jean Bourgain 
    \par   
    avec admiration et tristesse. 
  \end{dedication}


\begin{center}
	Annotation.
\end{center}

{\it \small
We prove in particular that for any sufficiently large prime $p$ there is $1\le a<p$ such that all partial quotients of $a/p$ are bounded by $O(\log p/\log \log p)$.
For composite denominators a similar result is obtained. 
This improves the  well--known Korobov  bound concerning Zaremba's conjecture from the theory of continued fractions.  
}
\\

\section{Introduction}

Let $a$ and $q$ be two positive coprime integers, $0<a<q$. 
By the Euclidean algorithm, a rational $a/q$ can be uniquely represented as a regular continued fraction
\begin{equation}\label{exe}
\frac{a}{q}=[0;c_1,\dots,c_s]=
\cfrac{1}{c_1 +\cfrac{1}{c_2 +\cfrac{1}{c_3+\cdots +\cfrac{1}{c_s}}}}
\,,\qquad c_s \ge 2.
\end{equation}

Assuming $q$ is known, we use $c_j(a)$, $j=1,\ldots,s=s(a)$ to denote the partial quotients of $a/q$; that is,
\begin{equation}\label{def:aq} 
    \frac aq := [ 0; c_1(a),\ldots,c_{s}(a)].
\end{equation}

Zaremba's famous conjecture \cite{zaremba1972methode} posits that there is an absolute constant $\zk$ with the following property:
for any positive integer $q$ there exists $a$ coprime to  $q$ such that in the continued fraction expansion (\ref{exe}) all partial quotients are bounded:
\[
c_j (a) \le \zk,\,\, \quad \quad 1\le j  \le s = s(a).
\]
In fact, Zaremba conjectured that $\zk=5$.
For large prime $q$, even $ \zk=2$ should be enough, as conjectured by Hensley \cite{hensley_SL2}, \cite{hensley1996}.
This theme is rather popular especially at the last time, see, e.g., 
papers 	
\cite{bourgain2011zarembas}--\cite{hensley1996}, 
\cite{KanIV}, \cite{Mosh_A+B}, 
\cite{MS_Zaremba}, \cite{Nied},  \cite{s_Chevalley}  and many others. 
The history of the question can be found, e.g.,  in \cite{Kontorovich_survey}, \cite{Mosh_survey}, \cite{NG_S}. 
We just notice here a remarkable progress of Bourgain  and Kontorovich \cite{bourgain2011zarembas}, \cite{BK_Zaremba} who proved Zaramba's conjecture for ``almost all'' 
denominators 
$q$.

Zaremba's conjecture is connected with some questions of  numerical integration. It was showed in \cite{zaremba1966} that if  Zaremba's conjecture is true, then 
the two--dimensional winding of the torus 
\[
	X = X(a,q) = \left\{ \left( \frac{j}{q}, \frac{aj}{q} \right) \right\}_{j=1}^{q} 
	\subseteq [0,1]^2 
\]
would 
have the least discrepancy (up to some absolute constants). 
Here we assume that the fraction $a/q$ enjoys $c_j (a) = O(1)$. 
In this direction, using some exponential sums,  Korobov \cite{Korobov_book} in 1963 proved that for any  prime  $q$  there is  $a$, $(a, q)=1,$ such that 
\begin{equation}\label{f:Korobov_log}
    \max _{\nu} c_{\nu}(a) \ll \log q \,.
\end{equation}
The same result takes place for composite $q$, see  \cite{Ruk}.

\bigskip 

In this paper we improve Korobov's bound \eqref{f:Korobov_log}. 
The proof is not purely analytical and uses rather well--known 
methods 
connected with the Bourgain--Gamburd machine \cite{BG} 
as well as 
an exact  result from \cite{Mosh_A+B}, see Lemma \ref{l:A+B} below.

\begin{theorem}
    Let $q$ be a positive 
    sufficiently large 
    integer 
    with sufficiently large prime factors. 
    Then there is a positive integer $a$, $(a,q)=1$ and 
\begin{equation}\label{f:main_M}
        M= O(\log q/\log \log q)
\end{equation}
    such that 
\begin{equation}\label{f:main_expansion}
    \frac{a}{q} = [0;c_1,\dots,c_s] \,, \quad \quad c_j \le M\,, \quad \quad  \forall j\in [s]\,.
\end{equation}
    Also, if $q$ is a sufficiently large square--free number, then \eqref{f:main_M}, \eqref{f:main_expansion} take place.\\ 
    Finally, if $q=p^n$, $p$ is an arbitrary prime, then \eqref{f:main_M}, \eqref{f:main_expansion}  hold for sufficiently large $n$. 
\label{t:main}
\end{theorem}


Our paper is organized as follows. In Section \ref{sec:prime} we obtain Theorem \ref{t:main} for sufficiently large prime $q$ and in the next Subsection \ref{subsec:square-free} we prove this for {\it all} sufficiently large square--free numbers. The last Subsection \ref{subsec:general} contains 
some discussions of the difficulties, which do not allow to  obtain  Theorem \ref{t:main}  following 
the standard  Bourgain--Varj\'u \cite{BV} variant of the  Bourgain--Gamburd machine for general $q$.
Also,  we separately consider the case $q=p^n$ here ($n$ is a sufficiently large number and $p$ is a prime) and show that Theorem \ref{t:main} remains to be true for such $q$. 
Using  the specific of our problem, we combine the approach of  \cite{BG_p^n}, \cite{BV} with a more simple and more direct two--dimensional method  from \cite{RS_SL2} to obtain Theorem \ref{t:main} for general $q$. 
We should say that all sections are dependent and the complexity increases from part to part. 
In the appendix we obtain some results on large deviations for continued fractions with bounded partial quotients. 
Our Theorem \ref{t:LD_CF} from the appendix is required in the previous Subsection \ref{subsec:general} (as a particular two--dimensional  case) and 
maybe it is interesting in its own right as it improves some results of Rogers \cite{Rogers}.

\bigskip 

The signs $\ll$ and $\gg$ are the usual Vinogradov symbols. 
Let us denote by $[n]$ the set $\{1,2,\dots, n\}$.
All logarithms are to base $2$.

\section{The prime case}
\label{sec:prime}

In this section we obtain our main Theorem \ref{t:main} in the case of prime $q$ although all results excluding our driving Lemma \ref{l:T-action} take place for an arbitrary number  $q$. 
The required generalization of Lemma \ref{l:T-action} for general $q$ is discussed in Section \ref{sec:general}.

We 
start with 
a well--known lemma, see  \cite[Lemma 5, pages 25--27]{Korobov_book} or \cite[Section 9]{Mosh_A+B}. 
It says that, basically, the partial quotients of a rational number are controlled via the hyperbola $x|y| = q/M$.

\begin{lemma}
    Let $a$ be coprime with $q$ and  $a/q = [0;c_1,\dots,c_s]$. 
    Consider the equation
\begin{equation}\label{eq:M_crit}
    ax \equiv y \pmod q \,, \quad \quad 1\le x<q \,, \quad  1\le |y|<q \,.
\end{equation}
    If for all solutions $(x,y)$ of the equation above one has $x|y| \ge q/M$, then $c_j \le M$, $j\in [s]$. 
    On the other hand, if for all $j\in [s]$ the following holds $c_j \le M$, then all solutions $(x,y)$ of \eqref{eq:M_crit} satisfy $x|y|\ge q/4M$. 
\label{l:M_crit}
\end{lemma}

Let $1\le t\le \sqrt{q}$ be a real number. 
Having a rational number  $\frac{a}{q} = [0;c_1,\dots,c_s]= \frac{p_s}{q_s}$, we write $\frac{p_\nu}{q_\nu}$ for its $\nu$-th convergent. 
Define 
\begin{equation}\label{def:Z_M(t)}
    Z_M (t) = \left\{ \frac{a}{q} = [0;c_1,\dots,c_s] ~:~ c_j \le M,\, \forall j\in [\nu],\, q_\nu < t \right\} \,.
\end{equation}
Also, put 
$$
    Q_M (t) = \left\{ \frac{u}{v} = [0;c_1,\dots,c_s] ~:~ c_j \le M,\, \forall j\in [s],\, v < t \right\} \,,
$$
and 
$$
    \overline{Q_M (t)} = \left\{ \frac{u}{v} = [0;c_1,\dots,c_s] \in Q_M (t)  ~:~ K( c_1,\dots,c_s, 1 ) \ge t \right\} \,,
$$
where by $K(d_1,\dots,d_k)$ we have denoted the correspondent continuant, see \cite{Hinchin}.
The sets $\overline{Q_M (t)}$ and $Z_M (t)$ are closely connected to each other, see \cite{Mosh_A+B}.

\bigskip 

To formulate further results we need  a definition from the real setting. 
Let  $M\ge 1$ be an integer.
Consider the set of {\it real} numbers  $F_M$, having all partial quotients  bounded by $M$. It is well--known \cite{Hinchin}, 
that for any  $M$ the Lebesgue measure of the set  $F_M$ is zero and its Hausdorff dimension  $w_M:=\mathcal{HD} (F_M)$ is  $w_M = 1-O(1/M)$, as  $M\to \infty$. Good bounds and asymptotic formulae on $w_M$ are contained in papers \cite{hensley1989distribution}---\cite{hensley1992continued}. 
The following result is a combination of Lemma 2 and Lemma 3 of \cite{Mosh_A+B}, as well as \cite[Theorem 2]{hensley1989distribution}. With some abuse of the notation we denote by the same letter $Z_M(t)$ the set of the {\it numerators} $a\in [q]$, $(a,q)=1$ from \eqref{def:Z_M(t)}.

\begin{lemma}
    Let $t \le \sqrt{q}$. 
    Then for some absolute constants $c_1,c_2>0$ one has 
\[
    Z_M (t) = B_1 \bigsqcup \dots \bigsqcup B_T \,,
    \quad \quad c_1 t^{2w_M} \le T\le c_2 t^{2w_M} \,,
\]
where $B_j$ are some disjoint  intervals and for all $j\in [T]$ the following holds $[q/t^2] \le |B_j|$.
\label{l:A+B}
\end{lemma}
\begin{proof}
In \cite{Mosh_A+B} it was proved in particular, that $T=|\overline{Q_M (t)}|$ and $[q/t^2] \le |B_j|$. 
Thus it remains to estimate the size of the set $\overline{Q_M (t)}$.

By \cite[Theorem 2]{hensley1989distribution}  we know that there exist absolute  positive constants $C_1, C_2$  such that
\begin{equation}\label{1}
C_1 t^{2w_M}
\le
|Q_M (t)|
\le
C_2 t^{2 w_M}
\end{equation}
for any $t\ge 2$.
Clearly, every  $u/v\in  {Q_M {(t)}}$ can be written as a continued fraction
\begin{equation}\label{2}
\frac{u}{v} = [0;A_1,..,A_l] \,\,\,\,
{ \text{with}\,\,\,\,\, A_l \ge 2} \,.
\end{equation}
The upper bound is obvious from the inclusion of 
$ \overline{Q_M(t)} \subset Q_M(t)$.
To prove the  lower bound put 
$$ k = \left(\frac{2C_2}{C_1}\right)^{\frac{1}{2w_M}}
$$
and consider the set 
$$
\mathcal{W} =  Q_M (t)\setminus Q_M (t/k) \,.
$$
By (\ref{1}) we see that 
$$
    |\mathcal{W}| \ge \frac{C_1}{2}\, t^{2w_M} \,.
$$
Any $u/v \in \mathcal{W}$ 
can be written in the form (\ref{2}) but we need another representation
\begin{equation}\label{3}
\frac{u}{v} = [0;A_1,\dots,A_l-1, 1] \,.
\end{equation}
Recall that 
\begin{equation}\label{4}
v = K(A_1,\dots,A_l-1, 1)  = K(A_1,\dots, A_l) < t \,.
\end{equation}
We define $\nu \ge 1$ from the condition
$$
K(A_1,\dots,A_l-1, \underbrace{1,\dots,1}_{\nu+1})<t\,\,\,\,\,\,
\text{but}\,\,\,\,\,\,
K(A_1,\dots,A_l-1, \underbrace{1,\dots,1}_{\nu+2})\ge t \,.
$$
As $ K(uw) > K(u) \cdot K(w)$ and $ t/k \le v < t$ we have
$$
K(\underbrace{1,\dots,1}_{\nu})<\frac{t}{v}\le k \,,
$$
and so
\begin{equation}\label{n}
\nu \le C_4 \log k \,.
\end{equation}
It is clear that 
\begin{equation}\label{5}
[0;A_1,\dots, A_l-1, \underbrace{1,\dots, 1}_{\nu+1}] \in \overline{Q_M {(t)}} \,.
\end{equation}
Each element ${u}/{v}\in \overline{Q_M{(t)}}$, which can be written in the form (\ref{5}) with continued fraction  (\ref{3}) satisfying (\ref{4}),
by (\ref{n}) can be written in such a form  not more than in $C_4 \log k$ ways.
So we have the bound
$$
|\overline{Q_M (t)}| \ge 
\frac{|\mathcal{W}|}{C_4\log k} \ge  C_5 
t^{2 w_M}\,\,\,\,
\text{with}\,\,\,\, C_5 = \frac{C_1}{2C_4\log k} 
$$
as required. 
$\hfill\Box$
\end{proof}

\bigskip 

The last result is actually contained in \cite[Proposition 7]{NG_S}.


\begin{lemma}
    Let $p$ be a prime number, $A,B \subseteq \F_p$ be sets, and $J=[N]$ be an interval.
    Then there is an absolute constant $\kappa>0$ such that 
\begin{equation}\label{f:T-action}
    |\{ (a+c)(b+c) = 1 ~:~ a\in A,\, b\in B,\, c \in 2\cdot J \}| - \frac{N|A||B|}{p} \ll \sqrt{|A||B|} N^{1-\kappa} \,.
\end{equation}
\label{l:T-action}
\end{lemma}

Lemma \ref{l:T-action} can be deduced from \cite[Proposition 7]{NG_S} directly.
The proof of 
\cite[Proposition 7]{NG_S} itself
is just an application of 
 the Bourgain--Gamburd machine \cite{BG} based on Helfgott's expansion result \cite{H}.
This method is rather well--known.
However we prefer to recall the main ideas and crucial steps of the   argument because we use them in the next Section \ref{sec:general}. So we are giving a sketched proof below.

\bigskip 

{\it Sketch of the proof of Lemma \ref{l:T-action}.} 
We use the notation $S(x)$ for the characteristic function of a set $S$. 
Also, write any $c\in 2\cdot J$ as $c=2j$, $j\in [N]$. 
Then clearly, the equation from the left--hand side of \eqref{f:T-action} is equivalent to $a=g_j b$, $j\in [N]$, where
$a\in A, b\in B$
\begin{equation}\label{def:g}
g_j =\left(\begin{array}{cc}
-2j & 1-4 j^2 \\
1 & 2j
\end{array}\right) \,,
\quad \quad j\in [N] 
\end{equation}
with $\det (g_j) = -1$. 
In \cite[Lemma 13]{NG_S} we considered the set of matrices  
\begin{equation}\label{def:G} 
G=\left\{\left(\begin{array}{cc}
1 & -2 j \\
2 j & 1-4 j^{2}
\end{array}\right): 1 \leq j \leq N\right\} \subset \SL_2 (\F_p) \,,
\end{equation}
and proved that the girth of the Cayley graph of $G$ (e.g., see the definition of the Cayley graph in Section \ref{sec:general} below) is at least $\tau \log_{N} p$, $\tau = 1/5$ for all sufficiently large $p$. 
The proof uses the well--known fact that $\SL_2 (\Z)$ contains the free subgroup, generated by 
$$
u= 
\left(\begin{array}{cc}
1 & 2 \\
0 & 1
\end{array}\right) 
\quad \quad 
    \mbox{ and } 
\quad \quad 
v= 
\left(\begin{array}{cc}
1 & 0 \\
2 & 1
\end{array}\right) \,.
$$
Then $G = \{ v^{j} u^{-j} ~:~ j\in [N] \}$ and it is easy to check that $G$ generates a free subgroup of $\SL_2 (\F_p)$ of rank $N$. 
 For any set $S\subseteq \SL_2 (\F_p)$ write $r_{S,2m} (x)$ for the number solutions to the equation 
 $$
 r_{S,2m} (x) := |\{ (s_1,\dots, s_{2m}) \in S^{2m} ~:~ s_1 s_2^{-1} s_3 \dots s^{-1}_{2m} = x\}| 
 =
 $$
 $$
 =
 \sum_{x_1 x_2^{-1} x_3 \dots x^{-1}_{2m} = x} S(x_1)S(x_2) S(x_3)\cdots S(x_{2m})
 \,.$$ 
 The same  sum
  $$
 \sum_{x_1 x_2^{-1} x_3 \dots x^{-1}_{2m} = x} f(x_1)f(x_2)f(x_3)\cdots f(x_{2m})
 $$ 
 can be defined for any function $f: \SL_2 (\F_p) \to \mathbb{R}$.
 Also, let $\T_{2m} (S) = \sum_x r^2_{S,2m} (x)$, see the discussion 
 concerning 
 these important quantities in \cite{TV} and in \cite[Sections 5, 6]{s_non_survey}. 
 After that one can apply the first stage of the Bourgain--Gamburd machine \cite{BG} to the set $G$, see \cite[Lemma 12]{NG_S}, which asserts that for any $g\in \SL_2 (\F_p)$ and an arbitrary proper subgroup $\Gamma < \SL_2 (\F_p)$ one has 
 \begin{equation}\label{cond:B-G} 
    \sum_{x\in g\Gamma} r_{G,2m} (x) \le \frac{|G|^{2m}}{K (G)} \,, 
  \end{equation}
where $m=\tau/4 \cdot \log_{N} p$ and $K (G) = p^{\tau/6}$. 
The quantity $K (G) \ge 1$ can be defined as the maximal one such that bound \eqref{cond:B-G} takes place (again it is possible to consider $K(f)$ for any non--negative function $f$).
Here one can use the symmetrization of $G$, considering $G\cup G^{-1}$ instead of $G$ as the authors did in \cite{BG} and in \cite{NG_S}, or apply the argument directly as was done in \cite[Section 6, see Theorem 49, Corollary 50]{s_non_survey}.
Further several applications of H\"older inequality (see \cite[Lemma 11]{NG_S}) or \cite[Lemma 32]{s_non_survey} (here the author considered a non--symmetric case but this is not important for further results) give us for an arbitrary 
function $f: \SL_2 (\F_p) \to \mathbb{R}^{}$, a positive integer $l$,   and any sets $A,B \subseteq \F_p$ that
 \begin{equation}\label{f:bound1} 
    \left| \sum_s \sum_{x\in B} f (s) A(sx) - \frac{ |A||B|}{p} \sum_s f(s) \right|    
    \le 
        \sqrt{|A||B|} \cdot \left( |B|^{-1} \sum_s r_{f,2^l} (s) \sum_{x\in B} B(sx) \right)^{1/2^l} \,. 
\end{equation} 
    More importantly, Helfgott's expansion result \cite{H} (see \cite[Propositions 5, 7]{NG_S}) allows us to estimate the quantity $\T_{2^k} (f)$ (for any sufficiently large $k$) and hence the right--hand side of \eqref{f:bound1}
    (it corresponds to the second and to the third stages of the Bourgain--Gamburd machine). 
    More precisely, it gives us that for any 
    function $F: \SL_2 (\F_p) \to \mathbb{R}^{}$ and a set  $B \subseteq \F_p$ the following holds 
 \begin{equation}\label{f:bound2} 
  \sum_s F(s) \sum_{x\in B} B(sx)
  \ll 
  |B| \|F\|_1 p^{-\d} \,,
\end{equation} 
    where $\d=1/2^{k+2}$ and $k\ll \frac{\log p}{\log K (f)}$, see details in \cite{NG_S} and in  \cite[Section 6, Theorem 49]{s_non_survey} (actually, one needs to use the balanced functions in formulae \eqref{f:bound1}, \eqref{f:bound2}). 

    To prove our lemma 
    we apply the first bound \eqref{f:bound1} with $f(x)=G(x)$ and the maximal $l$ such that $2^l \le 2m$.  
    After that we use the second estimate  \eqref{f:bound2} with $F(x) = r_{f,2^l} (x)$.
    Thanks to \eqref{cond:B-G} we know that $K(F) = K(r_{f,2^l}) \ge  p^{\tau/6}$.  
    Hence recalling that $m=\tau/4 \cdot \log_{N} p$,  and putting $\delta=\delta(\tau) = \exp(-C/\tau)$, where $C>0$ is an absolute constant,   we derive 
\[
    \sum_s \sum_{x\in B} G (s) A(sx) - \frac{ |A||B| |G|}{p}
    \ll 
    \sqrt{|A||B|} |G| p^{-\d/24 m} 
    \ll 
    \sqrt{|A||B|} N^{1-\kappa} \,,
\]
    where $\kappa>0$ is another absolute constant. 
%
 Thus we have obtained bound \eqref{f:T-action} for the set $G$. 
 As for our initial family of maps \eqref{def:g}, then, of course the multiplication of  $G$ by any element of $\mathrm{GL}_2 (\F_p)$ does not change the energy $\T_k$ and hence everything remains to be true for  the set defined in \eqref{def:g}. 
An alternative (but essentially equivalent) way to obtain the required  result is to show that all non--trivial representations of the non--commutative Fourier transform of the characteristic function of  $G$ enjoy 
an 
exponential saving,
see \cite[Corollary 50]{s_non_survey}.
This completes the scheme of the proof of our lemma. 
$\hfill\Box$

\bigskip 

Now we are ready to prove Theorem \ref{t:main} in the case of prime $q$. 
Take a parameter $\eps \in (0,1/2]$, which we will choose later and let $t=q^{1/2-\eps}$.
We assume that $t=o(\sqrt{q})$, $q\to \infty$ and hence we have the condition 
\begin{equation}\label{cond:eps_p}
    \eps \gg  \frac{1}{\log q} \,.
\end{equation}
Let $\mathcal{B}=\{0,1,\dots,c q/t^2-1\} = [0,1,\dots,cq^{2\eps}-1]$, where $c = \min\{ c_1/(4c_2), 1/4\}$. 
Then for a certain set of shifts $\mathcal{A}$ and a set $\Omega$, $|\Omega|\le |\mathcal{B}| T \le c c_2 q^{2\eps} t^{2w_M}$ one has 
\begin{equation}\label{f:dec_Z}
    Z_M:= Z_M (t) = (\mathcal{B}+ (\mathcal{B}\dotplus \mathcal{A})) \bigsqcup \Omega = (\mathcal{B}+Q) \bigsqcup \Omega = \tilde{Z}_M \bigsqcup \Omega  \,.
\end{equation} 
We have $|Z_M|\ge c_1 q^{2\eps} t^{2w_M}/2$ and hence 
$|\tilde{Z}_M| \ge |Z_M|/2$. 
Let $J$ be the maximal interval such that $2\cdot J \subset \mathcal{B}$. Thus $N:=|J| \ge |\mathcal{B}|/4$. 
Using Lemma \ref{l:T-action}
(recall once again that $q$ is a prime number and thus one can apply this lemma) 
with $A=B=Q = \mathcal{B}\dotplus \mathcal{A}$ and $J=J$, we obtain for a certain absolute constant $C>0$ that 
\begin{equation}\label{f:eq_J}
|\{ (a+i)(b+i) = 1 ~:~ a,b \in Q,\, i\in 2\cdot J \}| \ge \frac{N|Q|^2}{q} - C |Q| N^{1-\kappa} 
\ge \frac{N|Q|^2}{2q} > 0 \,.
\end{equation}
To satisfy the last inequality, we need the condition $|Q|N^\kappa \gg q$.
In other words, we must have
\begin{equation}\label{f:eq_J_cond_q}
    q^{2\eps(1+\kappa -w_M)} \gg q^{1-w_M} 
\end{equation}
or, equivalently, (recall that $1-w_M \sim  1/M$) 
\begin{equation}\label{cond:eps}
    \eps \gg \frac{1}{M} \,.
\end{equation}
Returning to \eqref{f:eq_J} and using decomposition \eqref{f:dec_Z}, we see that there are $z_1,z_2\in \tilde{Z}_M \subseteq Z_M$ with $z_1z_2 \equiv 1 \pmod q$.  
Put $a=z_1$.
In view of Lemma \ref{l:M_crit} we have that for all  $x\le t$ and $1\le |y| <q$ with $ax\equiv y \pmod q$ one has $x|y| \ge q/4M$.
Now we recall a well--known fact that the continued fractions are connected with the question of finding the inverse $a^{-1}$ 
modulo $q$, see \cite{Hinchin}. 
More precisely, we have 
\begin{equation}\label{f:inverse1}
\frac{a^{-1}}{q}=\left[0 ; c_{s}, c_{s-1} \ldots, c_{1}\right]
\quad \quad \quad \quad 
\text {if }  s \text { is even }
\end{equation}
\begin{equation}\label{f:inverse2} 
\frac{a^{-1}}{q}=\left[0 ; 1, c_{s}-1, c_{s-1} \ldots, c_{1}\right] \quad \text { if } s \text { is odd. }
\end{equation} 
Thus 
in view of formulae \eqref{f:inverse1}, \eqref{f:inverse2} for any  $x\le t$ and $1\le |y| <q$ with $a^{-1} x\equiv y \pmod q$ one has $x|y| \ge q/4M$.
The last modular equation is equivalent to 
$x \equiv ya \pmod q$ and hence any solution of \eqref{eq:M_crit} 
satisfy  
$$
    x|y| \ge \frac{q}{4M} \quad \quad 
    \mbox{for} \quad \quad x\in [t] 
    \quad \quad \mbox{ and } \quad \quad 
    x\in  \left[ \frac{q}{4Mt}, q \right) \,.
$$
Putting $t=\sqrt{q/4M}$ we see by Lemma \ref{l:M_crit}  that all partial quotients of $a/q$ are bounded by $4M$. 
Since $t=q^{1/2-\eps}$, it follows that $2 M^{1/2} = q^\eps$ or, equivalently, $\eps \sim \log M/\log q$.
We need to satisfy conditions \eqref{cond:eps_p} and  \eqref{cond:eps}.
Hence it is enough to have 
$$
    M \log M \gg \log q 
$$
as required. 
$\hfill\Box$

\bigskip

    
    Let us make one more remark. 
    In \cite[Theorem 3]{s_BG} it was proved
\begin{theorem}
    Let $p$ be a prime number, $\delta \in (0,1]$, 
    $N\ge 1$ be a sufficiently large integer, $N\le p^{c\delta}$ for an absolute constant $c>0$, 
    $A,B\subseteq \F_p$ be sets, and $g\in \SL_2 (\F_p)$ be a non--linear map. 
    Suppose that $S$ is a set, 
    $S \subseteq [N]\times [N]$, $|S| \ge N^{1+\delta}$.
    Then there is 
    a
    constant $\kappa = \kappa (\d) >0$ such that 
\begin{equation}\label{f:BG_new}
    |\{ g(\a+a) = \beta+b ~:~ (\a,\beta) \in S,\,a\in A,\, b\in B \}|
    - \frac{|S||A||B|}{p}
    \ll_g \sqrt{|A||B|} |S|^{1-\kappa} \,.
\end{equation}
\label{t:BG_new}
\end{theorem}

    Taking $S= [N]\times [N]$, $\d=1$ and $gx=1/x$, we 
    get 
    an analogue of Lemma \ref{l:T-action} for the correspondent  two--dimensional family of modular transformations.
    This more flexible method gives an alternative way to obtain our main Theorem \ref{t:main} in the prime case.


\section{The general case}
\label{sec:general}

We need some definitions, which will be used in this section. 
By $\pi_n$ denote the canonical projection modulo $n$. 
Having 
a matrix 
$$
g= 
\left(\begin{array}{cc}
\a & \beta \\ 
\gamma & \d
\end{array}\right) = (\a \beta | \gamma \d) 
\in \mathrm{Mat}_2 (\mathbb{R}) 
$$
we write $\| g\|$ for $\sqrt{\a^2+\beta^2+\gamma^2+\d^2}$.
The same can be defined for $\mathrm{Mat}_d (\mathbb{R})$, $d>2$.  
Recall that given an arbitrary set  $A\subseteq \Gr$ in a group $\Gr$ one can define the {\it Cayley graph} $\mathrm{Cay} (\Gr,A)$ with the vertex set $\Gr$ and a pair $(x,y)\in \Gr \times \Gr$ forms an edge iff  $yx^{-1} \in A$.
Having a probability measure $\nu$ on $\SL_d (\mathbb{R})$ (that is, a non--negative function with $\sum_x \nu(x) = 1$),  let us define the {\it top Lyapunov exponent} 
\begin{equation}\label{def:la_1}
    \lambda_1 (\nu) = \lim_{n\to \infty} \frac{1}{n} \int \log \|g\|\,  d r_{\nu,n} (g) \,,
\end{equation}
where we have assumed that  $\int \log \|g\|\,  d\nu (g)< \infty$, say (below our measures $\nu$ are supported onto a  finite number of matrices and hence this condition trivially takes place).
Basically, we are working in $\SL_2$ and hence we do not need  higher Lyapunov exponents (obviously, the second one is $-\la_1 (\nu)$).

Now to consider the general case of an arbitrary composite $q$ we naturally  require a theory of the growth in 
$\SL_2 (\Z/q\Z)$ or, even more generally, in $\SL_d (\Z/q\Z)$ with  $d>2$ due to we want to obtain an appropriate generalization of Lemma  \ref{l:T-action}. 
The question on the growth 
was considered in \cite{BG_p^n}, \cite{BV}, \cite{MOW_Schottky} and also in \cite{BGT_approx}, \cite{PS_Lie}. 
For example, let us formulate an application of this 
technique, see \cite{BV}.

\begin{theorem}
  Let  $S \subset \SL_d (\Z)$ be a finite and symmetric set. Assume that $S$ generates a subgroup $G< \SL_d (\Z)$  which is Zariski dense in  $\SL_d$.\\
  Then $\mathrm{Cay} (\pi_q (G), \pi_q (A))$ form a family of expanders, when $S$ is fixed and $q$ runs through the integers. Moreover, there is an integer $q_0$ such that $\pi_q (G) = \SL_d (\Z/q\Z)$ if $q$ is coprime to $q_0$. 
\label{t:BV}
\end{theorem}

It is well--known \cite{Tits} that if $S$ generates a subgroup $G$ which is  Zariski dense in  $\SL_d$, then $G$ contains a subgroup $\Gamma$, which is free on two generators and is Zariski dense in $\SL_d$. 
All calculations in \cite{BG_p^n}, \cite{BGS},  \cite{BV} concern this smaller free group $\Gamma$. 
Roughly 
speaking, in our proofs we check that these calculations remain  to be true for 
the 
set $G$ from \eqref{def:G}, which generates a free subgroup of rank $N$. 
For simplicity, we start with the case of square--free $q$ where the required theory of the  growth in $\SL_2 (\Z/q\Z)$ is more concrete. 
The general case will be considered in Subsection \ref{subsec:general} and 
our discussion follows paper \cite{BV}  
(notice that, actually, the proof in \cite{BV} even does not 
suppose that the number of generators is exactly two), as well as \cite{BG_p^n} and \cite{RS_SL2}. 
Finally, notice that the condition of Theorem \ref{t:BV} that $q$ coprime to $q_0$ says, basically, that all prime divisors of $q$ must be sufficiently large.

\subsection{The square--free case} 
\label{subsec:square-free}

In this subsection 
let $q$ be a sufficiently large square-free number and we want to obtain an analogue of Theorem \ref{t:main}, that is we want to find  a positive $a$ such that $(a,q)=1$ and 
$$
    \frac{a}{q} = [0;c_1,\dots,c_s] \,, \quad \quad c_j \le M\,, \quad \quad  \forall j\in [s]\,, 
$$
where     
$$
        M= O(\log q/\log \log q) \,.
$$
In this case the general scheme of the proof remains the same 
(of course one should replace $q$ in 
\eqref{f:eq_J}, \eqref{f:eq_J_cond_q} by $q^{1+o(1)}$ because we consider $\Z^*_q$ but not just $\Z_q$, 
anyway 
condition \eqref{cond:eps} does not change)
and to prove the required analogue of Lemma \ref{l:T-action} for square--free $q$ we need the crucial result of paper \cite[Proposition 4.3]{BGS}.

\begin{theorem}
    Let $q$ be a square--free number, $q=\prod_{p\in \mathcal{P}} p$. 
    Also, let $A\subset \SL_2 (\Z/q\Z)$ be a set, $\kappa_0,\kappa_1 >0$ be constants such that 
    $q^{\kappa_0} < |A| < q^{3-\kappa_0}$, further
\begin{equation}\label{cond:BGS1}
    |\pi_{q_1} (A)| > q_1^{\kappa_1}\,, \quad \quad \forall q_1 | q\,, \quad \quad q_1 > q^{\kappa_0/40} \,,
\end{equation}
    and for all $t\in \Z/q\Z$, for any $b\in  \mathrm{Mat}_2 (q)$ with $\pi_p (b) \neq 0$, $\forall p\in \mathcal{P}$ we have 
\begin{equation}\label{cond:BGS2}
    |\{ x\in A ~:~ \mathrm{gcd} (q, \mathrm{Tr} (bx)- t ) > q^{\kappa_2} \}| = o(|A|) \,,
\end{equation} 
    where $\kappa_2 = \kappa_2 (\kappa_0, \kappa_1) > 0$. 
    Then 
\begin{equation}
    |A^3| > q^{\kappa} |A| 
\end{equation}
    with $\kappa = \kappa (\kappa_0, \kappa_1) >0$. 
\label{t:BGS_SL2}
\end{theorem}

One of the pleasant features of Theorem \ref{t:BGS_SL2} is that it does not require the knowledge of the subgroup lattice of $\SL_2 (\Z/q\Z)$ (which is rather complex for square--free numbers  $q$ although, of course $\SL_2 (\Z/q\Z) \simeq \prod_{p\in \mathcal{P}} \SL_2 (\Z/p \Z)$ by the Chinese remainder theorem). 

\bigskip 

Now to obtain Lemma \ref{l:T-action} for square--free numbers we apply the usual  Bourgain--Gamburd machine as in the previous section and we use the notation of it as well. 
The only thing we need to check is that for any $z\in \SL_2 (\Z/q\Z)$ the product $zP_*$ of the set 
\begin{equation}\label{def:P_*}
    P_* = \{ x \in \SL_2 (\Z/q\Z) ~:~ \D < r_{G,2l} (x) \le 2\D \} 
\end{equation}
satisfies all conditions of Theorem \ref{t:BGS_SL2}, see the proof of \cite[Theorem 49]{s_non_survey} or Theorem \ref{t:BV_growth} below.
Here $l\ge m = \tau/4 \cdot \log_{N} q$ (see Section \ref{sec:prime})  and $\D$ is a positive number such that
\begin{equation}\label{cond:P_*}
    \D |P_*| \ge \frac{|G|^{2l}}{K_*} \,,
\end{equation}
where $K_* = q^\eps$ for a certain small $\eps>0$.  
Notice that $P_*$ is a symmetric set (although it is not really important for us). 

To check all conditions of Theorem \ref{t:BGS_SL2} we, basically, repeat the calculations from  \cite[pages 595--599]{BGS}. 
Condition \eqref{cond:BGS1} 
follows 
rather quickly. 
Indeed, take any $q_1 | q$ such that $q_1 > q^{\kappa_0/40}$ and choose $m_1\le m$ with $(5N^2)^{10m_1} \sim q_1$.
Also, notice that  $\max_{g\in G} \|g\| \le 5N^2$.  
Then $\pi_{q_1} : G^{2m_1} \to \SL_2 (\Z/q\Z)$ is one--to--one.
In view of \eqref{cond:P_*}, we obtain 
\begin{equation}\label{tmp:16.07_1}
    \frac{|G|^{2l}}{K_*} \le \D |P_*| \le \sum_{x\in P_*} r_{G,2l} (x) \le |G|^{2l-2m_1} \max_{w\in \SL_2 (\Z/q\Z)} \sum_{x\in wP_*}  r_{G,2m_1} (x) \,,
\end{equation}
and hence by the well--known Kesten result \cite{Kesten} on random walks, we have for  $w\in \SL_2 (\Z/q\Z)$ maximizing \eqref{tmp:16.07_1} that 
\begin{equation}\label{tmp:05.07_1}
    |wP_* \cap \supp (G^{2m_1})| \ge 
    \frac{|G|^{2m_1}}{K_*} \cdot (2|G|-1)^{-m_1} \,.
\end{equation}
    Using the last bound, we get  
\begin{equation}\label{f:pi_P_*}
    |\pi_{q_1} (zP_*)| = |\pi_{q_1} (wP_*)| \ge 
    |wP_* \cap \supp (G^{2m_1})| 
    \ge 
    \frac{|G|^{m_1}}{2^{m_1} K_*} 
    \gg K^{-1}_* q^{1/40}_1 
    = K^{-1}_* q^{\kappa_0/1600}
\end{equation}
as required (let $\eps \le \kappa_0/3200$ and $\kappa_1 = \kappa_0/5000$, say).

Further notice that  we can easily assume that $q^{\kappa_0} < |P_*| = |zP_*| < q^{3-\kappa_0}$. 
Indeed, if $|P_*| \ge q^{3-\kappa_0}$ for sufficiently small $\kappa_0$ (actually, the bound $|P_*| \ge q^{2+\zeta}$ for any $\zeta>0$ in enough), then 
one can 
apply a suitable  variant of the Frobenius Theorem \cite{Frobenius} (an appropriate adaptation to the square--free case can be found in \cite[pages 587--588]{BGS} or in \cite[Lemma 7.1]{BG_p^n}, also see \cite[Theorem 49]{s_non_survey}).
The inequality $|P_*| > q^{\kappa_0}$ is also almost immediate. 
Indeed, as $l\ge m$ we have by the Kesten bound as above in \eqref{tmp:05.07_1} 
\[
   \frac{|G|^{2l}}{K_*} \le \D |P_*| \le \sum_{x\in P_*} r_{G,2l} (x) 
   \le 
   |P_*| (2 |G|)^m |G|^{2l-2m}
\]
and hence
$|P_*| \ge (|G|/2)^{m} K^{-1}_* \ge q^{\tau/4} 2^{-m} K^{-1}_*  \gg q^{1/80} K^{-1}_*$ and choosing sufficiently small $\eps$ one can take  $\kappa_0 = 1/100$, say.

Now 
it remains to 
check the property \eqref{cond:BGS2} and here we use calculations from  \cite[pages 597--599]{BGS}.  
It is sufficient to show that for all $t\in \Z/q\Z$, for any  $b\in  \mathrm{Mat}_2 (q)$, $\pi_p (b) \neq 0$, $\forall p\in \mathcal{P}$, and for all $q_2|q$ satisfying $q_2 > q^{\kappa_2}$, we have 
\begin{equation}\label{f:BGS2_ref}
    |\{ x\in zP_* ~:~ \mathrm{Tr} (gx) \equiv t  \pmod {q_2}  \}| \le  q^{-\epsilon} |P_*| 
\end{equation} 
for a certain $\epsilon >0$. 
Let us choose $m_2$ such that $(5N^2)^{16m_2} \sim q_2$. 
Assuming that \eqref{f:BGS2_ref} fails, we derive as in \eqref{tmp:16.07_1} that for a certain $w\in \SL_2 (\Z/q\Z)$ one has 
\begin{equation}\label{tmp:18.07_1}
    \sum_{ x\in G^{2m_2} ~:~ \mathrm{Tr} (bwx) \equiv t  \pmod {q_2}} r_{G,2m_2} (x) \ge |G|^{2m_2} K_*^{-1} q^{-\epsilon} \,.
\end{equation}
Clearly, for $b':= bw$ one has $\pi_p (b') \neq 0$, $p\in \mathcal{P}$. 
Let $T \subseteq G^{2m_2}$ be the set of $x$ from   \eqref{tmp:18.07_1}. 
It is easy to see that for any $x\in T$ one has $\| x\| \le (5N^2)^{2m_2}$ and that the set $T$ is a hyperspace in our four--dimensional vector space $\mathrm{Mat}_2 (q)$ equipped with the standard inner product  $\langle A, B  \rangle := \mathrm{Tr} (AB^*)$. 
Then for any $x^{(1)},x^{(2)}, x^{(3)}, x^{(4)}, x\in T$, we derive for an arbitrary $p|q_2$ that 
\[
f(x^{(1)},x^{(2)}, x^{(3)}, x^{(4)}, x)
:=
\]
\begin{equation}\label{f:det_hsp} 
\mathrm{det}
\left(\begin{array}{cccc}
x^{(1)}_{11} - x_{11} & x^{(2)}_{11} - x_{11} & x^{(3)}_{11} - x_{11} &  x^{(4)}_{11} - x_{11} \\
x^{(1)}_{12} - x_{12} & x^{(2)}_{12} - x_{12} & x^{(3)}_{12} - x_{12} &  x^{(4)}_{12} - x_{12} \\
x^{(1)}_{21} - x_{21} & x^{(2)}_{21} - x_{21} & x^{(3)}_{21} - x_{21} &  x^{(4)}_{21} - x_{21} \\
x^{(1)}_{22} - x_{22} & x^{(2)}_{22} - x_{22} & x^{(3)}_{22} - x_{22} &  x^{(4)}_{22} - x_{22} 
\end{array}\right) 
\equiv 0 \pmod p \,.
\end{equation} 
Clearly, the determinant above does not exceed $15\cdot 2^{19} (5N^2)^{8m_2} < q_2$, say,  and hence this determinant is just zero in $\Z$.
Whence it is zero modulo any prime number and we choose a prime $P$ such that 
\begin{equation}\label{f:choice_P}
\log P \sim 2m_2 \cdot \log N    
\end{equation} 
(in \cite{BGS} the number $P$ was just $\log P \sim 2m_2$ and this choice corresponds to the fixed number of generators, that is, $N=O(1)$ here).
Notice that 
\begin{equation}\label{tmp:P_below}
    P \ge \exp(\Omega (m_2 \log N)) \ge q^{\Omega (1)}_2 \ge q^{\Omega (\kappa_2)} \,.
\end{equation}
Let us estimate $\pi_P (T)$ from below.
It will allow us to obtain a lower bound for 
the number of the solutions to equation \eqref{f:det_hsp} modulo $P$ as $|\pi_P (T)|^5$.
One the other hand, there is a universal Weil--type 
upper bound (even a rough estimate works)  
for the number of the solutions to the polynomial equation $f(x^{(1)},x^{(2)}, x^{(3)}, x^{(4)}, x) \equiv 0 \pmod P$ with variables in $\SL_2 (\Z/P\Z)$ and having the form $O(P^{14})$, see details and the required references in \cite[page 599]{BGS}. It will give the 
desired 
contradiction and hence the 
demanded 
bound \eqref{f:BGS2_ref} takes place.

Thus it requires  to estimate $\pi_P (T)$ from below.
By the previous section, that is, by the expansion result in $\SL_2 (\Z/P\Z)$ we know that in this group one has  $r_{G,2m_2} (x) \ll |G|^{2m_2}/P^3$, thanks to our choice of $P$ (and $m_2$). 
Returning to calculations in  \eqref{tmp:18.07_1} 
and using the last bound, we get  
\begin{equation}\label{f:pi_P(T)}
    |\pi_{P} (T)|\cdot |G|^{2m_2}/P^3 \gg |G|^{2m_2} K_*^{-1} q^{-\epsilon} 
\end{equation}
and hence $|\pi_{P} (T)| \gg P^3 K_*^{-1} q^{-\epsilon}$.
Thanks to \eqref{tmp:P_below} it gives us at least $P^{15} K_*^{-5} q^{-5\epsilon} \gg P^{14}$ solutions to equation \eqref{f:det_hsp} modulo $P$ (here $\epsilon$ and $\eps$ are sufficiently small numbers)  
 and this is a contradiction. 
As we have seen from the proof the square--free case is reduced to the prime case, eventually.

Again an alternative way of the proof is to use
the girth--free result \cite[Theorem 3]{s_BG} and work with the two--dimensional family of modular transformations.  
$\hfill\Box$


\subsection{Discussion and completion of the proof}
\label{subsec:general}

As we have seen in the previous subsection the result for square--free $q$ can be derived from an appropriate version of the Helfgott growth theorem in $\SL_2 (\F_p)$, see \cite{H} and \cite{BG}. 
The growth result in $\SL_2 (\Z/p^n \Z)$ follows a similar scheme (combining with a deep but independent sum--product theorem in $\Z/q\Z$, see \cite{B_Zq} plus some additional ideas, of course), that is, it follows from the growth result for prime $P$, see \cite[formulae (4.2), (4.3) and Proposition 4.2]{BG_p^n}. 
As in \eqref{f:choice_P} we chose $P$ as $\log N \cdot 2m_2 \sim \log P \gg \log q$, where $q = p^n$ and thus the parameter $l\sim m_2$ in \cite[see estimates (3.8), (3.9), (4.2) and further formulae]{BG_p^n}  is now $l\sim \log_N P$ but not just $\log P$. 
Once again, it matches  with the calculations of the previous subsection and reflects the fact that now we have $N$ free generators instead of $O(1)$ and all of them have norm at most $5N^2$ but not $O(1)$. 
Hence we obtain Theorem \ref{t:main} for $q=p^n$ for all sufficiently large primes $p$ and $n$ rather easily. On the other hand, for small $p$ the result follows from the well--known Folding lemma \cite{Nied}.


\begin{lemma} 
  Let $\tilde{q}\ge 2$ be an integer. Then  for any positive integer  $n$
  there exists $a_n, (a_n,\t{q}) = 1$ such that in the continued fraction expansion
  $$
  \frac{a_n}{\t{q}^n} = [0; c_1,\dots,c_s]
  $$
  all partial quotients are bounded by $c_j \le \t{q}^2-1$, $j \in [s]$.
\end{lemma} 
\begin{proof}   
  We use the argument from Niederreiter \cite{Nied} based on the Folding lemma (see \cite{Mosh_survey,PS}). It is clear that the result is true for $n = 1,2$.
  Suppose that a positive integer  $Q$ can be represented via  a continuant
  \begin{equation}\label{continuant}
  Q = K(c_1,\dots ,c_{t-1},c_t)   = K( c_t,c_{t-1},\dots ,c_1) =  K(1, c_t-1,c_{t-1},\dots,c_1)\,, \quad \mbox{where} \quad  c_j \ge 2 
  \end{equation}
  with bounded elements $c_j \le M$, $j\in [t]$.
  By the Folding lemma  for any  positive integers $c_j$ and $X$ we have the equality
  \begin{equation}\label{fold-}
  K(c_1,\dots, c_{t-1},c_t, X, 1, c_t-1,c_{t-1},\dots, c_1) 
\end{equation} 
  \begin{equation}\label{fold}
  = K(c_1,\dots,c_{t-1},c_t) \cdot K(1, c_t -1 , c_{t-1},\dots ,c_1) (X +1)= Q^2 (X+1) \,.
  \end{equation}
  Let  $Q = \t{q}^n$.
  Clearly, the continuant in \eqref{fold-} has elements bounded by $\max (M,X)$.
  Choosing $ X = \t{q}-1$ and $X=\t{q}^2-1$ and combining formulae  \eqref{continuant} and \eqref{fold}, we obtain 
  representations of $\t{q}^{2n+1}$ and $\t{q}^{2n+2}$ via continuants with elements bounded by $\max(M, \t{q}^2-1)$.
  Consider the sets
  $$
  A_1 = \{1,2\} \quad \mbox{and} \quad A_{n+1} =A_n \cup \{ 2n+1, 2n+2: \,\,\, x \in A_n\} \quad  \mbox{for} \quad n \ge 1 \,.
  $$
  Now $\bigcup_{n=1}^\infty A_n$ is the set of all positive integers and the result follows. 
 $\hfill\Box$
\end{proof}

\bigskip 

In the general case the argument  \cite{BV}, which allows to obtain Theorem \ref{t:BV}, say, is 
different 
and it 
based (besides deep consideration of \cite{BV}, of course)
on very strong tools from 
\cite{BFLM}. 
%
Let us recall 
the driving result on the growth in $\SL_d (\Z/Q\Z)$, see \cite[Proposition 2]{BV}.

\begin{theorem}
    Let $G\subset \SL_d (\Z)$ be a symmetric finite set, $G$ generates a group $\Gamma$ which is Zariski--dense in $\SL_d$. Then for any $\eps>0$ there is $\d>0$ such that the following hold. If $P' \subseteq \Gamma$ is a symmetric set and $l$, $Q$, $(Q,q_0)=1$ are sufficiently large integers satisfying 
\begin{equation}\label{f:BV_growth}
    \sum_{x\in P'} r_{G,l} (x) > \frac{|G|^{l}}{Q^{\d}} \,, \quad 
    l>\d^{-1} \log Q \quad and \quad |\pi_Q (P')| < Q^{3-\eps} \,, 
\end{equation}
    then $|(P')^3| > |P'|^{1+\d}$.
    Here  $q_0$ is a fixed positive integer (depending on $G$). 
\label{t:BV_growth}
\end{theorem}

We need to check conditions \eqref{f:BV_growth} for a shift $zP_*$ of our set $P_*$ from \eqref{def:P_*}, \eqref{cond:P_*} and the set $G$ is the same as in \eqref{def:G} (clearly, $G$ generates a (semi)group $\G$ which is Zariski--dense in $\SL_d$).
But thanks to assumption \eqref{cond:P_*} one can see that the first condition of \eqref{f:BV_growth} trivially takes place (with $l=2l$ and $K_* = Q^\delta$), further as we have discussed before the third 
assumption 
follows from the Frobenius Theorem (again, an appropriate adaptation for general $Q$ can be found in \cite[Page 5]{BV} and in \cite[Lemma 7.1]{BG_p^n}). 
Also, thanks to the Pl\"unnecke--Ruzsa inequality \cite{R_PR_ineq} (or see \cite{TV}) it is easy to check that the growth of our symmetric set $P'$, namely, $|(P')^3| > |P'|^{1+\d}$ implies the growth of  any of its shift $|(zP')^3| > |P'|^{1+c'\d}$, where $c'>0$ is an absolute constant (just consider $zP' (zP')^{-1} zP' = z(P')^3$). 
Thus we can think below that $z$ is the identity and thus we can work with the set $P_*$ solely.
The only thing we need to check  is the second condition $l>\d^{-1} \log Q$, which must be replaced to $l>\d^{-1} \log_N Q$.
Then formula \cite[estimate (3)]{BV} obviously works, as well as the proof of Proposition 3, page 9 of the same paper due to the fact that this 
proposition 
requires to consider just the square--free case, which was 
obtained 
in the previous subsection.  
Also, notice that the constant $C(d,L)$ from the proposition  remains to be constant under this choice of $l$ as calculations \cite[page 9]{BV} show and this is important for us.

Theorem \ref{t:BV_growth} 
follows from the combination of 
Proposition 3 and Proposition 6 of \cite{BV}. 
Thus it remains to check that the choice $l>\d^{-1} \log_N Q$ does not change Proposition 6 in our particular case. Here the authors use a deep result from \cite{BFLM} and we formulate a convenient consequence of it (see \cite[Theorem A]{BFLM} and \cite[Theorem B, Lemma 7]{BV}).

\begin{theorem}
Let $S \subset \SL_d (\Z)$ be a symmetric set, $S$ generates a subgroup $\Gamma < \SL_d (\Z)$ which acts 
proximally and   strongly irreducibly on $\mathbb{R}^d$.
Assume further that any finite index subgroup of $\G$  generates the same $\mathbb{R}$--subalgebra of $\mathrm{Mat}_d (\mathbb{R})$  as $\G$.\\
Then there is a constant $c_0>0$ 
 for any $a, b \in \Z^d \setminus \{ 0 \}$, $a$ is coprime to $q$ we have
\begin{equation}\label{f:BFLM}
    |S|^{-l} \sum_g e^{\frac{2\pi i \langle g a, b \rangle}{q} } \, r_{S,l} (g)
    \ll (q/\mathrm{lcm} (q,b))^{-1/C}
\end{equation}
    for $l \gg \max\{ \la^{-1}_1 (\nu) \cdot \log q, \log q \}$. Here the measure $\nu$ is $\nu(x) = S(x)/|S|$. 
\label{t:BFLM}
\end{theorem}

The proof of Theorem \ref{t:BFLM} based on the theory of products of random matrices \cite{RM_book_BFLM}, \cite{RM_book}, \cite{Furstenberg_RM} and in particular, on the large deviations for the Lyapunov exponents, see \cite[Theorem 4.3]{BFLM}.
It is easy to calculate the top Lyapunov exponent $\la_1 (\nu)$ in our two--dimensional case, namely, $\la_1 (\nu) \sim \log N$ (and as we said before $\la_2 (\nu) = -\la_1 (\nu)$) see, e.g., formula \eqref{f:LD_CF_double} below.  
Further one problem with \cite[Theorem 4.3]{BFLM} is that all bounds here depend on $\nu$ (and hence on $N$). Again, in our  two--dimensional case everything can be calculated effectively thanks to reducing the problem to classical ergodic theorems with the Gauss shift $T$, see estimate \eqref{f:LD_CF_double} of Theorem \ref{t:LD_CF} from the appendix.  
Nevertheless, the dependence on $N$ in \cite{BFLM} does not allow to get the required bound for $l$ (basically, due to the fact that the large deviations bounds do not use the circumstance that the top Lyapunov exponent $\la_1 (\nu) \sim \log N$ is growing) and we leave the possibility of it as an open 

\bigskip

{\bf Question.} Is it possible to obtain Theorem 
\ref{t:BFLM}
with $l \gg \log_N Q$ for our concrete set $G$ of two--dimensional matrices? If so, it would allow to obtain 
another proof of 
Theorem \ref{t:main} for all $q$ with sufficiently large prime factors. 

\bigskip 

Anyway at the moment we cannot use a rather general technique from paper \cite{BFLM}. Instead of this we restrict ourselves to the case $d=2$ and  follow the scheme of the proof  \cite[Theorem 5]{RS_SL2},  as well as \cite[Section 4]{BG_p^n}, which we have already discussed above. 

\bigskip

The following simple lemma is a slight generalization of Exercise 1.1.8 in \cite{TV}. 

\begin{lemma}
    Let $\Gr$ be a group and $A,B \subseteq \Gr$ be sets. 
    Then there exists a set $X\subseteq ABB^{-1}$ with 
\[
    |X| \ll \frac{|ABB^{-1}|}{|B|} \cdot \log |AB| 
\]
    such that $AB\subseteq XB$. 
\label{l:random_shift}
\end{lemma}

Now let us obtain the following ``escaping'' result for our set $P_*$. 
Actually, it is a small 
modification 
of \cite[Lemma 4.1]{BG_p^n} and we almost repeat the proof of it in the  particular case of a linear function $f(g) = \mathrm{Tr} (wg)$, $w\in \SL_2 (\Z/q\Z)$
(also, see calculations \eqref{tmp:16.07_1}, \eqref{f:det_hsp} of the previous subsection).
As above we identify $\mathrm{Mat}_2 (\mathbb{Z})$ with $\mathbb{Z}^4$, e.g., for $g_1,g_2,g_3,g_4 \in \mathrm{Mat}_2 (\mathbb{Z})$ by $(g_1,g_2,g_3,g_4)$ we denote the correspondent $4\times 4$ matrix.

\begin{lemma}
    Let $q_*$ be a divisor of $q$, $P_*$ be a set as in \eqref{def:P_*}, satisfying \eqref{cond:P_*} and let 
    $r>0$ be an integer. 
    Suppose that $|P^3_*| = K|P_*|$.
    Also, let $f(g)$
    be a linear 
    function on $\SL_2 (\Z)$
    in $4$ variables, which does not vanish  identically on $\SL_2 (\Z)$. 
    Then 
\begin{equation}\label{f:escaping}
     |\{ g \in P^r_*  ~:~ f(g) \equiv 0 \pmod {q_*} \}|
    \ll_f 
        K^{2r^2} K_* \log^r q   \cdot \frac{|P_*|}{q_*^{c}} \,,
\end{equation}
    where $c>0$ is an absolute constant. 
\label{l:escaping} 
\end{lemma}
\begin{proof} 
    In view of Lemma \ref{l:random_shift}, as well as the Pl\"unnecke--Ruzsa inequality \cite{R_PR_ineq} (or see \cite{TV}) we can
    split the set $P^r_*$ as $XP_*$, where $|X| \ll K^{2r^2} \log^r q$.  
    Thus it is 
    enough to obtain \eqref{f:escaping} for any set of $g$ in $z P_*$, where $z\in \SL_2 (\Z/q\Z)$ and after that sum up  all bounds. 
    Further as in \eqref{tmp:16.07_1}, \eqref{f:det_hsp}  it is sufficient to consider the case $f(g) = 0$ (the equality in $\Z$) and then the case   $f(g) \equiv 0 \pmod {q_*}$ will easily follow if we take $l_* = c_* \log_N q_*$, where $c_* >0$ is a sufficiently small constant and consider just $2l_*$-th power of $G$, see below. 
    Fix $z$ and denote by $S=S_{z}$ the set of $g\in P_*$ with $f(z g) \equiv 0 \pmod {q_*}$. 
Then by the definition of the set $P_*$ one has for a certain new $z' \in \SL_2 (\Z/q\Z)$ 
\begin{equation}\label{tmp:S}
    |G|^{2l-2l_*} 
    \sum_{g\in S} r_{G,2l_*} (z' g) 
    \ge 
    \sum_{g\in S} r_{G,2l} (g) 
    \ge |S| \D \,.
\end{equation}
    Here we have used that $l \ge  m = \tau/4 \cdot \log_N q$ 
    and thus we can assume that $l_* \le l$. 
    %
    %
    Now recall that $f$ is a linear function on $\SL_2 (\Z)$.  
    In other words, in the space $\mathrm{Mat}_2 (\mathbb{Z})$ equipped with the inner product $\langle \cdot, \cdot \rangle$, we have for a certain $w\in \SL_2 (\Z/q\Z)$ and $C\in \Z/q\Z$ that  $f(g) =
    \mathrm{Tr} (wg) + C = 
    \langle w,g^* \rangle + C$.
    Taking $g_1,\dots,g_5 \in S$ 
    which take part in the first summation from \eqref{tmp:S}, 
    we 
    get  $\mathrm{Tr} (wz'g_j)+C \equiv 0 \pmod {q_*}$, $j\in [5]$ and hence
\[
    \langle g_1 - g_2, (wz')^* \rangle = \dots = \langle g_1 - g_5, (wz')^* \rangle \equiv 0 \pmod {q_*} \,.
\]
    Considering the adjoint  matrix, we see that 
$\det (g_1- g_2, \dots, g_1- g_5) \cdot (wz')^* \equiv 0 \pmod {q_*}$. 
But $(wz')^* \in \SL_2 (\Z/q\Z)$, further $q_*$ is a divisor of $q$ by our assumption  and hence 
$\det = \det (g_1,\dots,g_5) := \det (g_1- g_2, \dots, g_1- g_5) \equiv 0 \pmod {q_*}$. 
    Clearly, $|\det| \le 4! 2^4 (5N^2)^{2l_*}$ and the last quantity can be done strictly less than  $q_*$ by our choice of the constant $c_*$ in the definition of  $l_*$.   
    Thus  $\det =0$ in $\Z$. 
    Choose a prime $P$ 
    similarly to  \eqref{f:choice_P}, \eqref{tmp:P_below} such that $\log P\sim l_* \log N$.
    Clearly, we have $\det \equiv 0 \pmod P$.
    By a Weil--type bound as in the previous subsection we have seen that the number of the solutions to the equation is $O_{f} (P^{14})$. 
    Now by the expansion result in $\SL_2 (\Z/P\Z)$ (see  \cite{H})  we know that in this group one has  $r_{G,2l_*} (x) \ll |G|^{2l_*}/P^3$, thanks to our choice of $P$.
    As in \eqref{f:pi_P(T)} and in \eqref{tmp:S}, we have 
\begin{equation}\label{f:pi_P(T)+}
    |\pi_{P} (S)|\cdot |G|^{2l_*}/P^3 \gg
    |S| \D |G|^{2l_* -2l} \,.
\end{equation}
    By our condition  \eqref{cond:P_*} and our choice of the parameter $l_*$, we have 
    (compare with estimate \eqref{tmp:P_below}) 
\[
    (|S| P^3 K^{-1}_* |P_*|^{-1})^5 
    \le (|S| \D P^3 |G|^{-2l})^5 \ll |\pi_P (S)|^5 \ll_f P^{14}
\]
    and hence 
\[
    |S| \ll_f K_* |P_*| P^{-1/5} \ll_f K_* |P_*| q_*^{-c} \,, 
\]
    where $c>0$ is an absolute constant. 
%
%
This completes the proof. 
$\hfill\Box$
\end{proof}

\bigskip

Now we are ready to obtain Theorem \ref{t:main} and as we have discussed above it is enough to prove $|P^3_*| > |P_*|^{1+\d}$ for the set 
$P_*$ from \eqref{def:P_*}, 
which satisfies 
\eqref{cond:P_*}. 
We write $K=|P^3_*|/|P_*|$ and our task is to obtain a good lower bound for $K$. 
As we said before we follow the argument of \cite{RS_SL2} (with some modifications), which is an adaptation of the general scheme from \cite{H}. 
In particular, we avoid using  the deep sum--product results in $\Z/q\Z$ from \cite{B_Zq}.

Let $T =T_w$ be the centralizer of an element $w\in \SL_2 (\Z/q\Z)$, which we call a maximal torus by uniformity  
reasons  
(see the notation from \cite{BGT_approx},  \cite{H}, \cite{RS_SL2}, for semisimple elements in $\SL_d$ there is no difference between its  centralizers and maximal tori=maximum commutative subgroups). 
We say that $T$ is {\it involved} with $P_*$ if there are $p_1,p_2\in P_*$ such that $g:=p_1^{-1}p_2 \in T$ and $g\neq \pm I$ ($I$ is the identity matrix). We now conjugate $T$ with all elements of $P_*$, considering the union $\bigcup_{h\in P_*} hTh^{-1}$. If all maximal tori $T'=hTh^{-1}$, arising thereby, are involved with $P_*$, we continue conjugating each of these tori with elements of $P_*$. After that, once again, either we get at least one new torus, which is not involved with $P_*$, or all the tori, generated so far from $T$ are involved with $P_*$. And so on. 
As we have discussed above 
the set $P_*$ generates $\SL_2 (\Z/q\Z)$ and since, the procedure will end in one of the two ways:  either (i) there is some torus $T$ involved with $P_*$ and  a certain  $h\in P_*$, such that $T'=hTh^{-1}$ is not involved with $P_*$, or (ii) for all $h\in \SL_2 (\Z/q\Z)$ and some (initial maximal torus) $T$, every torus $hTh^{-1}$ is involved with $P_*$.
Consider the two scenarios separately.

\medskip
{\sf Case (i) -- pivot case.}  
 
The maximal torus $T'$ is not involved with $P_*$. However,  $T=h^{-1}T'h$ is:  there is a non-trivial
element $g\in P_*^{-1} P_*=P_*^2$ (here we have used that $P_*= P_*^{-1}$ but it is not really important), lying in $h^{-1}T'h$, therefore $g'=hgh^{-1}\in T'$.
Consider the projection 
$$
\_phi: \, P_*\to C_\tau,\;\;\; p_* \to p_* g' p_*^{-1}\in P^6_*\,,
$$
where $C_\tau$ is the conjugacy class of $g$ with $\mathrm{Tr} (g) = \tau$. 
This projection is at most two-to-one, for if $h_1,h_2$ have the same image, this means that $h_1^{-1}h_2\in T'$, $h_1,h_2 \in P_*$,  but $T'$ is not involved with $P_*$, thus $h_1^{-1}h_2=\pm I$. It follows that $|P_*^6\cap C_\tau|\geq |P_*|/2$. 
Write $P_{**} = P^6_* \cap C_\tau$.
Our task is to find a good upper bound for $P_{**}$ of the form $|P_{**}| \ll |P^6_*|^{1-\eps_0}$, where $\eps_0>0$ is an absolute constant.  After that the required lower bound for $K$ will follow from the Pl\"unnecke--Ruzsa inequality.

Consider the function $\Cf (y) = |P_{**} \cap y^{-1} P_{**}|$. 
By the Cauchy--Schwarz inequality, we have
\begin{equation}\label{tmp:E_5}
    |P_{**}|^{4} \le \sum_{y} \Cf^2 (y) \cdot |P_{**} P^{-1}_{**}| \,.
\end{equation}
For any $g\in P_{**} \cap y^{-1} P_{**}$ one has 
 $\tau = \mathrm{Tr} (g) = \mathrm{Tr} (y g)$. 
Applying Lemma \ref{l:escaping} with $q_*=q$, 
$r=12$ and the following non--vanishing 
linear function 
$f(y) = \mathrm{Tr} (yg)-\tau$, we have in view of \eqref{tmp:E_5} 
\begin{equation}\label{f:p_case1}
    |P_{**}|^{4} \le |P_{**}|^2 |P^{12}_*|
    \cdot 
    K^{288} K_* \log^{12} q  \cdot \frac{|P_*|}{q^c}
\end{equation} 
and hence thanks to $|P_{**}|\ge |P_*|/2$ and the Pl\"unnecke--Ruzsa inequality, we get 
\[
    q^{c/2} \ll K_* K^{300} \,.
\] 
Recall that $K_* = q^\eps$ and thus if we take $\eps=c/20$, then 
one obtains 
$K \gg q^{c/1000}$, say.  It is absolutely enough for our purposes due to the fact that $P_*$ is large (see, e.g., calculations from \eqref{f:pi_P_*}).

\bigskip

 {\sf Case (ii) -- large set case.}  Suppose, for any $h\in G$, all tori $h T h^{-1}$ 
 are involved with $P_*$. The number of  such tori (not meeting, except at $\{\pm I\}$) 
 will be calculated in purely algebraic Lemma \ref{l:Stab_p^n} 
 and (as the worst case scenario) one may assume that $P_* P_*^{-1}\setminus \{\pm I\} = P_*^2 \setminus \{\pm I\}$ is partitioned between these tori.

%
Thus it follows by the Helfgott orbit--stabilizer Theorem \cite{H}, \cite[Lemma 11 and page 19]{RS_SL2}   that 
\begin{equation}\label{tmp:orb-stab}
K |P_*| \ge |P_*^2|  \ge \sum_{h\in \SL_2 (\Z/q\Z)/N(T)} |P_*^2 \cap h T
h^{-1}| \gg \frac{|\SL_2 (\Z/q\Z)|}{|N(T)|} \cdot \frac{|P_*|}{|P_*^4\cap C_\tau|}\,,
\end{equation}
where
$N(T)$ is the normaliser of $T$.  
Similarly to above  (see calculations in \eqref{f:p_case1}) we estimate $|P_*^4\cap C_\tau|$ as $|P_*^4\cap C_\tau| \ll K^{150} K^{1/2}_* q^{-c/20} |P_*|$ 
(actually, before  we have considered six products instead of four and hence one can obtain even better bound). 
Now suppose that we have chosen our torus $T$ as 
\begin{equation}\label{c:T_r}
    |N(T)| \ll q^{1+\zeta} \,,
\end{equation}
    where $\zeta = \zeta(c)>0$ is a sufficiently small number. 
    One can see that the size of the normaliser of ``typical'' $T$ is $O(q)$ and hence bound \eqref{c:T_r} is close to the optimal. 
Then 
thanks to \eqref{tmp:orb-stab}, \eqref{c:T_r}, we obtain  
\[
    K^{151} K^{1/2}_*  |P_*| \gg q^{2+c/20-\zeta} \ge q^{2+c/40} \,, 
\]
where we have chosen $\zeta \le c/40$. 
Taking the parameter $\eps$ in $K_* = q^\eps$ to be $\eps = c/100$,
we see that 
either $K \gg q^{c/30000}$
or $|P_*| \gg q^{2+c/160}$. 
In the former case we are done and the last case was discussed before and follows from the Frobenius Theorem (again, an appropriate adaptation for general $q$ can be found in \cite[Page 5]{BV} and in \cite[Lemma 7.1]{BG_p^n}).

\bigskip

It remains to obtain an algebraic lemma to satisfy condition \eqref{c:T_r} and we use some ideas of paper \cite{BG_p^n}. 
Somehow we need to choose $T=T_w$ such that $w$ is ``far'' from the identity $\pm I$ (clearly, $|N(\pm I)| \sim q^3$ and hence \eqref{c:T_r} fails in this case). 
Below we assume that all primes  $p$ (they will be  divisors of $q$) are odd.
For any $g\in \SL_2 (\Z/p^n\Z)$ we write 
\begin{equation}\label{f:Bourgain_p^n}
    g = \frac{\mathrm{Tr}\, g}{2} I + p^{r(g)} \cdot (ab|c(-a)) = \frac{\mathrm{Tr}\, g}{2} I + p^{r(g)} \cdot g' \,,
\end{equation}
    where not all $a,b,c$ vanish modulo $p$. 
    Since $\det (g) \equiv 1 \pmod {p^n}$, we have 
\begin{equation}\label{f:tr_p^n}
    \left( \frac{\mathrm{Tr}\, g}{2} \right)^2 \equiv 1 + p^{2r(g)} (a^2 + bc) \pmod {p^n} \,,
\end{equation}
    and hence in particular, 
\begin{equation}\label{f:tr_p^n+}
\mathrm{Tr}\, g \equiv \pm 2 \pmod {p^{s_* (g)}}\,, 
\quad \quad 
\mbox{ where}  
\quad \quad 
s_* (g) =\min \{ n, 2r(g)\}
\,.
\end{equation} 
    Writing  $r=r(g)=r_p (g)$, we can calculate several algebraic characteristics of $g$ in terms of 
    $r$.

\begin{lemma}
    Let $g\in  \SL_2 (\Z/p^n\Z)$ and $r=r_p (g)$. 
    Then $|\Stab (g)| \le 8 p^{n+2r}$.
    Further $|N(\Stab (g))| \le 300 p^{n+3r}$.
\label{l:Stab_p^n} 
\end{lemma}
\begin{proof}
Taking $h \in \Stab (g)$ and using \eqref{f:Bourgain_p^n} with $t:=r_p (h)$ and 
$h' = (\a \beta| \gamma (-\a))$, we obtain 
\begin{equation}\label{f:Stab_1}
p^{t+r}
\left(\begin{array}{cc}
\a & \beta \\ 
\gamma & -\a
\end{array}\right) 
\left(\begin{array}{cc}
a & b \\ 
c & -a
\end{array}\right) 
\equiv 
\left(\begin{array}{cc}
a & b \\ 
c & -a
\end{array}\right) 
\left(\begin{array}{cc}
\a & \beta \\ 
\gamma & -\a
\end{array}\right) 
p^{t+r} 
\pmod {p^n} \,.
\end{equation}
We assume firstly that $t+r < n$ and write $q_1 = p^{n-t-r}$ and $q_2 = p^{n-t} \ge q_1$. 
Then we obtain from \eqref{f:Stab_1} the following system of equations 
\begin{equation}\label{f:sys_st1}
   \beta c \equiv \gamma b \pmod {q_1} \,, 
    \quad \quad 
    \a b \equiv a \beta \pmod {q_1} \,,
    \quad \quad 
    \gamma a \equiv \a c \pmod {q_1} \,. 
\end{equation}
Since not all $a,b,c$ vanish modulo $p$, it follows that there are exactly $q_1$ solutions to system  \eqref{f:sys_st1}. 
Hence we obtain 
\[
    3p^{3r} q_1 = 3p^{3r} p^{n-t-r} = 3p^{n-t+2r} \le 3p^{n+2r} 
\]
solutions to \eqref{f:Stab_1}. 
Returning to \eqref{f:tr_p^n}, \eqref{f:tr_p^n+} for $h$, we find $\mathrm{Tr}\, h$ solving the quadratic equation modulo $p^{s_* (h)}$ and then modulo $p^n$ (it gives two solutions) and hence by  \eqref{f:Bourgain_p^n} we reconstruct $h$. 
Now if $t+r \ge n$, then  we can take $\a,\beta, \gamma \in [q_2]$ in an arbitrary way and after that we reconstruct $\mathrm{Tr}\, h$ as above. It gives us at most 
\[
    2q^3_2= 2 p^{3n -3t } \le 2 p^{3r} \le 2p^{n+2r}
\]
solutions to \eqref{f:Stab_1}.

Now let us 
obtain 
that $|N(\Stab (g))| \le 300 p^{n+3r}$. 
Let 
$\mathrm{n}=(\a \beta|\gamma \d) \in N(T)$ and $h\in \Stab(g)$.
Suppose that (other cases can be considered in a similar way) in system \eqref{f:sys_st1}, we have $a \neq 0 \pmod p$ and hence   
\begin{equation}\label{f:h_Stab}
    h' = h'_t = \mu \cdot (1 b|c (-1))  + q_1 \cdot (AB|C(-A)) \,,  
\end{equation}
where $\mu \neq 0 \pmod p$ runs over $[q_1]$, $A,B,C$ run over $[p^r]$ and $b$, $c \in [q_1]$ are some new fixed elements.   
Having the condition 
\begin{equation}\label{tmp:h^n}
    \mathrm{n}^{-1} h \mathrm{n} \equiv  \frac{\mathrm{Tr}\, h_t}{2} I + p^{t} \cdot  \mathrm{n}^{-1} h'_t \mathrm{n} \equiv  
    \frac{\mathrm{Tr}\, h_{t_1}}{2} I + p^{t'} \cdot \mathrm{n}^{-1} h'_{t_1} \mathrm{n}  \pmod {p^n} \,,
\end{equation}
we clearly, derive $\mathrm{Tr}\, h_t = \mathrm{Tr}\, h_{t_1}$ and thanks to \eqref{f:h_Stab} one can see that $t=t'$.
Further identity \eqref{tmp:h^n} holds for all $t$ and in particular for $t=0$. 
Using \eqref{f:h_Stab} for this choice of $t$ (it gives us $t'=0$ and $q_1 = p^{n-r}$), we see that  
\[
\left(\begin{array}{cc}
\d & -\beta \\ 
-\gamma & \a
\end{array}\right) 
\left(\begin{array}{cc}
1 & b \\ 
c & -1
\end{array}\right) 
\left(\begin{array}{cc}
\a & \beta \\ 
\gamma & \d
\end{array}\right) 
\equiv 
\la
\left(\begin{array}{cc}
1 & b \\ 
c & -1
\end{array}\right) 
\pmod {q_1} 
\,,
\]
where $\la  \neq 0 \pmod p$ is a number. 
The last equation is equivalent to the system modulo $q_1$
\begin{equation}\label{f:3eq_N}
    \a (\d - \beta c) + \gamma (\d b +\beta ) = \la \,,
    \quad 
    2\beta \d - \beta^2 c + \d^2 b  = b \la\,,
    \quad 
    -2\a \gamma +\a^2 c - \gamma^2 b = c \la \,.
\end{equation}
Solving the second equation in \eqref{f:3eq_N}, which is a  non--vanishing quadratic equation, we obtain at most $2q_1$ solutions.
Now combining the first equation of \eqref{f:3eq_N} with another linear equation in $\a,\gamma$, namely, with $\a \d - \beta \gamma \equiv 1 \pmod {q_1}$, we find the only solution  in $\a,\gamma$ unless $\d (1-\la) = \beta c$, and $-\d b = \beta (1+\la)$. 
If the last equation has the only solution in $\d,\beta$, then we have at most $2q_1$ solutions in $\a,\gamma$
(it follows from $\a \d - \beta \gamma \equiv 1 \pmod {q_1}$ or from the third equation of system \eqref{f:3eq_N}). 
Otherwise $bc = \la^2-1$. 
Here we have used the fact that 
either $\beta$ or $\d$ is invertible  modulo $p$.
Applying this (without loss of generality we  assume that $\d$ is invertible), 
as well as 
the third equation from \eqref{f:3eq_N}, combining with  $\a \d - \beta \gamma=1$, we derive 
\begin{equation}\label{tmp:18.07_2}
    \gamma^2 (c \beta^2 - b \d^2 - 2\beta \d) + 2\gamma (\beta c-\d) + c-c \d^2 \la = 0 \,.
\end{equation}
If the last quadratic equation is trivial modulo $p^s$ for a certain $s$, then we have $\d = \beta c$, $c\beta^2 (1+bc) = c \beta^2 \la^2 = 0$ and $c=c^3 \beta^2 \la =0$. 
Here we have used that $bc = \la^2-1$ and $\la \neq 0 \pmod p$. 
Hence $\d=0$ and returning to \eqref{f:3eq_N}, we see that $\la = -1$, $b=\a = 0$.
If $p^s = q_1$, then from $\a \d - \beta \gamma \equiv 1 \pmod {q_1}$, we see that there are at most $q_1$ solutions in $\beta, \gamma$. 
If $p^s < q_1$, then there exists at most two solutions in $\gamma$ of equation \eqref{tmp:18.07_2} and we reconstruct $\a$ from the third equation of \eqref{f:3eq_N}, say, in at most two ways. 
Thus we have in total at most $9\cdot 2^5 q_1$ solutions modulo $q_1$ and hence we obtain at most $9\cdot 2^5 p^{n+3r}$ solutions  modulo $p^n$.
This completes the proof of the lemma. 
$\hfill\Box$
\end{proof}

\bigskip

Finally, it remains to choose an appropriate initial torus $T$, satisfying condition \eqref{c:T_r}. 
Let $q=p^{\a_1}_1 \dots p^{\a_s}_s$. 
Notice that, by the Chinese remainder theorem, we have 
$\SL_2 (\Z/q\Z) \simeq \prod_{j=1}^s \SL_2 (\Z/p^{\a_j} \Z)$. 
Consider the collection $\mathcal{C}$ of all divisors of $q$, having the size at least $q^{\kappa}$, where $\kappa = \kappa (\zeta) >0$ is a parameter. 
Clearly, by the divisor function bound one has $|\mathcal{C}| \le q^{o(1)}$. 
For any  $q' \in \mathcal{C}$ we apply Lemma \ref{l:escaping} with $q_*=q'$, $r=2$, $f_{\pm} (g) = \mathrm{Tr}\, (g) \pm 2$.
Thus  we either find an element $w\in P^2_*$ with $\mathrm{Tr}\, w \neq \pm 2 \pmod {q'}$, where $q'$ runs over $\mathcal{C}$ or 
\[
    K^{8} K_* \gg q^{\kappa c-o(1)} \,.
\]
Suppose that the later holds. 
Choosing 
the parameter $\eps$ in $K_* = q^\eps$ to be $\eps = \kappa c/4$,
we see that $K\gg q^{\kappa c/16}$ and we are done.
Now take our element $w$ and consider the following sets 
\[
    G = \{ j\in [s] ~:~ r_{p_j} (w) \le p^{\kappa_1 \a_j} \} \,, 
    \quad \quad B=[s] \setminus G \,,
\]
where $\kappa_1$ is another parameter. 
By \eqref{f:tr_p^n+} for any $j\in B$ one has either $\mathrm{Tr}\, w \equiv 2 \pmod {p_j^{\kappa_1 \a_j}}$, or $\mathrm{Tr}\, w \equiv -2 \pmod {p_j^{\kappa_1 \a_j}}$. 
Hence by our construction of the set $\mathcal{C}$, we have 
\begin{equation}\label{f:bad_j}
    \prod_{j\in B} p_j^{\kappa_1 \a_j} \le q^{2\kappa} \,.
\end{equation}
Using Lemma \ref{l:Stab_p^n}, we derive
\[
    |N(T_w)| \ll 
    \prod_{j\in G} p^{\a_j (1+3\kappa)}_{j} \cdot \prod_{j\in B} p^{3\a_j} \le 
    \prod_{j} p^{\a_j (1+3\kappa)}_{j} \cdot q^{2\kappa/\kappa_1}
    \le q^{1+3\kappa + 2\kappa \kappa^{-1}_1} \,.
\]
It remains to take $\kappa = \kappa^2_1$ and, say, $\kappa_1 = \zeta/8$.
Thus we have obtained  the required condition \eqref{c:T_r}. 
This completes the proof of Theorem \ref{t:main} for general $q$. 
$\hfill\Box$

\section{Appendix}

In this section 
we obtain the large deviations estimate for the top Lyapunov exponent of our set $G$ 
defined in 
\eqref{def:G}, namely, for the measure $G(x)/|G|$
(one can see that the top Lyapunov exponent is just 
$\lim_{n\to \infty} \frac{1}{n} \log q_n ([0; c_1,-c_1, \dots,c_n, -c_n])$, where $c_j \in 2\cdot [N]$, $j\in [n]$). 
Such bounds can be used in the proof of Theorem \ref{t:BFLM}, see \cite{BFLM} in the particular case of the group $\SL_d (\Z/q\Z)$ with $d=2$ (and for the concrete measure). 
For $\SL_2 (\Z/q\Z)$ the theory of products of random matrices can be replaced by the 
standard 
considerations from the theory of continued fractions.
Recall one more time that in the considered  regime the parameter $N$ 
tends to infinity. 
We hope that Theorem \ref{t:LD_CF} is interesting in its own right and even in the classical case, see formula \eqref{f:LD_CF} below.
At least we should mention that this inequality implies  the identity $g_N:= \lim_{n\to \infty} q^{1/n}_n ([0; c_1,\dots,c_n]) = N/e+o(N)$ for a.e. $(c_1,\dots,c_n) \in  [N]^n$, $N\to \infty$ and this is better than the estimates on $g_N$ from \cite[Lemma 3]{Rogers}, also
see discussion  \cite[pages 43--44]{Rogers}. 

In our proof we follow the method from \cite{FWSL}. Recall that by $T: [0,1] \to [0,1]$ we denote the Gauss shift, that is,  $Tx=\{1/x\}$ for $x\in (0,1]$ and $T0 =0$. 

\begin{theorem}
    Let $n,N$ be positive integers, $\d\in (0,1]$ be a real number.
    Then there are absolute constants $\kappa\in (0,1]$ and $K\ge 1$ such that for all $N\ge K \d^{-2} \log (1/\d)$ and $n\ge K \d^{-1} \log (1/\d) \log N$ one has 
\[
    N^{-n} \left| \left\{ (c_1,\dots,c_n) \in  [N]^n ~:~ 
    \left| \frac{1}{n} \log q_n ([0; c_1,\dots,c_n]) - \frac{\log (N!)}{N} \right| \ge \d \right\} \right|
\]
\begin{equation}\label{f:LD_CF}
    \le  2 \exp \left(-\frac{\kappa \d^2 n}{\log (1/\d)} \right) \,. 
\end{equation}
    Similarly, under the same conditions on $n$ and $N$ the following holds 
\[
    N^{-n} \left| \left\{ (c_1,\dots,c_n) \in 2\cdot [N]^n ~:~ 
    \left| \frac{1}{n} \log q_n ([0; c_1,-c_1, \dots,c_n, -c_n]) - \frac{2\log (N!)}{N} \right| \ge \d \right\} \right|
\]
\begin{equation}\label{f:LD_CF_double}
    \le  4 \exp \left(-\frac{\kappa \d^2 n}{\log (1/\d)} \right) \,. 
\end{equation}
\label{t:LD_CF}
\end{theorem}
\begin{proof} 
    Let $L=\log N$.
    Writing $X_j = [0; c_j,\dots,c_n]$ and
    applying 
    the well--known formula 
    $p_j (x) = q_{j-1} (Tx)$ for any $x\in [0,1]$,  we see that
\begin{equation}\label{f:q_n_1}
    q_n ([0; c_1,\dots,c_n]) := q_n (x) = \frac{q_n(x)}{p_n (x)} \cdot   \frac{q_{n-1} (Tx)}{p_{n-1} (Tx)} \dots \frac{q_1(T^{n-1} x)}{p_1 (T^{n-1} x)} 
\end{equation}
    (we have used that $p_1 (T^{n-1}x) = 1$) and hence 
\begin{equation}\label{f:q_n_2}
    q_n ([0; c_1,\dots,c_n]) = \prod_{j=1}^n X^{-1}_j \,.
\end{equation}
Thus it is sufficient to estimate the probability 
\[
    \mathbb{P}_{\d,[n]} := \mathbb{P} \left\{ \left| \frac{1}{n} \sum_{j=1}^n \log X_j + \frac{\log (N!)}{N} \right| \ge \d \right\}  \,.
\]
Notice that $J:=\frac{\log (N!)}{N}$ is close to the expectation of the random variable $\frac{1}{n} \sum_{j=1}^n \log X_j$. 
Indeed, using  
the standard 
estimates for continuants, we have by the stationarity 
\begin{equation}\label{f:error_in_expectation} 
    -\frac{1}{n} \sum_{j=1}^n \mathbb{E} \log X_j = N^{-2} \sum_{a,b=1}^N \log (a+ \theta_1 b^{-1}) = \frac{\log (N!)}{N} + \theta_2 \frac{\log^2 N}{N^2} \,,
\end{equation}
where here and below $|\theta_j| \le 1$ are some absolute constants.  
In \eqref{f:error_in_expectation} we have used the approximation 
\begin{equation}\label{f:X_j} 
    X^{-1}_j (c_j,\dots, c_n) = c_j + \frac{\theta_3}{c_{j+1}} \,. 
\end{equation}
Similarly, notice that 
\begin{equation}\label{f:X_j_*} 
    X^{}_j (c_j,\dots, c_n) = c^{-1}_j + \frac{\theta_4}{c_j^2 c_{j+1}} \,. 
\end{equation}
Now by our assumption we have $N\ge K \d^{-2} \log (1/\d)$  and hence the error in 
\eqref{f:error_in_expectation} is at most $\d/4$ for large $K$ and hence it is negligible. 
Also, let us remark that  by the Stirling formula one has 
\begin{equation}\label{tmp:02.07_1}
    N^{-1} \log(c_1 N) \le |J - (L-1)| \le N^{-1} \log(c_2 N)  \,,
\end{equation}
(here and below $c_j>0$ are some absolute constants).
Similarly, take any $0<s \le 1/2$ and using  the Euler--Maclaurin  formula (or just  a direct calculation) and formulae \eqref{f:X_j}, \eqref{f:X_j_*}, we  derive that 
\begin{equation}\label{tmp:02.07_2}
    \log \mathbb{E} |X_1|^{s} \le -s L - \log (1-s) + \frac{c_3L}{N}  \le -s L + s + s^2 + \frac{c_3 L}{N} 
    \,,
\end{equation}
as well as 
\begin{equation}\label{tmp:02.07_2'}
    \log \mathbb{E} |X_1|^{-s} \le s L - \log (1+s) + \frac{c_4L }{N} \le s L - s +  \frac{s^2}{2} + \frac{c_4 L}{N}  
    \,.
\end{equation}
Now let $4\le M \le n/4$ be an  
even 
parameter and we split $[n]$ into $M$ arithmetic progressions of size $t:=[n/M]$, namely, $Q_1,\dots, Q_M$ having the step $M$. 
Since the union of $Q_j$ is $[n]$ plus at most $M-1$ points, we can assume that $n$ is divisible by $M$ and hence $t=n/M$.
Indeed, it requires just to replace $\d$ in  $\mathbb{P}_{\d,[n]}$ to $\d/2$ and notice that 
$$
    2 M \max_j \| \log X_j\|_\infty \le 2M L \le \d n /4 \,,
$$
where the condition $n\ge 8ML/\d$ will be checked later.
Now we have $t=n/M$ and use the exponential Markov inequality with a parameter $\la>0$, $\la \le 1/(2M)$ and the H\"older inequality to derive
\[
    \mathbb{P}_{\d/2,[n]} \le \exp(-\d \la n/2 + \la n J) \cdot \mathbb{E} (\prod_{i=1}^n |X_i|^\la) 
    =
     \exp(-\d \la n/2 + \la n J) \cdot \mathbb{E} (\prod_{i=1}^M \prod_{j\in Q_i} |X_j|^\la) 
\] 
\begin{equation}\label{tmp:04.07_1}
    \le 
    \exp(-\d \la n/2 + \la n J) \cdot 
    \prod_{i=1}^M \left( \mathbb{E} \prod_{j\in Q_i} |X_j|^{\la M} \right)^{1/M} 
    \,.
\end{equation} 
Here we have considered  the case when $\frac{1}{n} \sum_{j=1}^n \log X_j - \frac{\log (N!)}{N}$ is positive and the opposite situation will be considered below in a similar way. 
Thus it remains  to estimate $\mathbb{E} \prod_{j\in Q_i} |X_j|^{\la M}$ for any $i\in [M]$. 
Using the well--known $\psi$--mixing property of our shift $T$ with $\psi (m) =C \mu^m$, where $C>0$ and $1/2 <\mu <1$ are some absolute constants, we get by the  stationarity  and the assumption $\la M \le 1/2$  (see details in \cite[Lemmas 2, 3]{FWSL}) that 
\[
    \mathbb{E} \prod_{j\in Q_i} |X_j|^{\la M}
    \le 
    (1+\psi(M/2))^{t} \left( \mathbb{E} |X_1|^{\la M} \right)^t \,.
\]
Substituting the last bound into \eqref{tmp:04.07_1} and using estimates \eqref{tmp:02.07_1}, \eqref{tmp:02.07_2}, we obtain 
for sufficiently large $N$, $L/N \ll (\la M)^2$  that 
\[
 \mathbb{P}_{\d/2,[n]}
 \le 
    \exp(-\d \la n/2 + \la n J + t \psi (M/2) - \la M t L + \la Mt + 2t (\la M)^2 )
\]
\[ 
    \le 
    \exp(-\d \la n/2 + n M^{-1} \psi (M/2) + 4n \la^2 M ) \,.
\]
    Now we choose $\la = \d/(16M)\le 1/(2M)$ and after that we take the parameter $M$ such that $M^{-1} \psi (M/2) \le \d \la /8 = \d^2/(128M)$.
    In other words, $\psi (M/2) \le \d^2/128$ and hence we can choose  $M\ll \log (1/\d)$. 
    It gives us
\[
 \mathbb{P}_{\d/2,[n]}
 \le 
    \exp(-\d \la n/8) = \exp(-\d^2 n/(128M)) =  \exp(-\kappa \d^2 n/\log (1/\d)) \,,
\]
    where $\kappa > 0$ is an absolute constant. 
    We need to check that  $n\ge 8ML/\d$ and $L/N \ll (\la M)^2 = 2^{-8} \d^2 $ but our assumptions $N\ge K \d^{-2} \log (1/\d)$, $n\ge K\d^{-1} \log (1/\d) \log N$ guarantee  it.

    Finally, let $\frac{1}{n} \sum_{j=1}^n \log X_j + \frac{\log (N!)}{N}<0$ and hence our exponential Markov inequality requires to estimate the probability 
\[
    \mathbb{P} \left\{ \exp \left( - \la \sum_{j=1}^n \log X_j \right) \ge \exp (n\la (J+\d/2)) \right\} \,.
\]
    In this case we use the same calculations, the same choice of the parameter $\la = \d/(16M)\le 1/2$, as well as  formulae \eqref{tmp:02.07_1}, \eqref{tmp:02.07_2'} to get for sufficiently large $N$ such that $L/N \ll (\la M)^2$  
\[
     \mathbb{P}_{\d/2,[n]}
 \le 
    \exp(-\d \la n/2 - \la n J + t \psi (M/2) + \la M t L - \la Mt + t (\la M)^2 )
\]
\[
    \le 
    \exp(-\d \la n/2 + n M^{-1} \psi (M/2) + 2n \la^2 M ) 
    \le 
    \exp(-\kappa \d^2 n/\log (1/\d))  
    \,.
\]

\bigskip 

It remains to obtain estimate \eqref{f:LD_CF_double}. 
As in \eqref{f:q_n_1}, \eqref{f:q_n_2} (recall that we assume that $M$ and hence $n$ are even numbers) we derive 
\[
    -\log q_n ([0; c_1,-c_1, \dots,c_n, -c_n]) := -\log q_n (x) = \sum_{j=1}^n \log Y_j (x) + \sum_{j=1}^n \log Z_j (x) \,,
\]
    where $Y_j = [0; c_j,-c_j, \dots, c_n,-c_n]$ and $Z_j =  [0;c_j,-c_{j+1},c_{j+1}, \dots, -c_n,c_n]$. 
    Thus it is sufficient to obtain the large deviation principle for the random variables $Y_j$, $Z_j$ separately. 
    Similarly to \eqref{f:X_j}, \eqref{f:X_j_*}, we have (recall that by the assumption $c_j\in 2\cdot [N]$) 
\begin{equation}\label{f:Y_j} 
    Y^{-1}_j (\omega)  = c_j + \frac{2\theta_1}{c_{j}} \,,
    \quad \quad \quad \quad 
     Y^{}_j (\omega) = c_j^{-1} + \frac{2\theta_2}{c^3_{j}} \,, 
\end{equation}
    and 
\begin{equation}\label{f:Z_j} 
    Z^{-1}_j (\omega)  = c_j + \frac{2\theta_3}{c_{j+1}} \,,
    \quad \quad \quad \quad 
     Z^{}_j (\omega) = c_j^{-1} + \frac{2\theta_4}{c^2_{j} c_{j+1}} \,.  
\end{equation}
    Thus we have the same asymptotic formulae for $-\frac{1}{n} \sum_{j=1}^n \mathbb{E} \log Y_j$, $-\frac{1}{n} \sum_{j=1}^n \mathbb{E} \log Z_j$ as in \eqref{f:error_in_expectation}. 
    Also, thanks to \eqref{f:Y_j}, \eqref{f:Z_j},  we get  \eqref{tmp:02.07_2}, \eqref{tmp:02.07_2'} for $Y_j, Z_j$. After that we repeat the calculation above and obtain the required estimate  \eqref{f:LD_CF_double}. 
This completes the proof. 
$\hfill\Box$
\end{proof}

\end{document}